\documentclass[11pt, a4paper, oneside, reqno]{amsart}
\usepackage{amsmath}

\usepackage{caption}
\usepackage{subcaption}

\usepackage[numbered]{bookmark}
\usepackage{changepage} 
\usepackage[utf8]{inputenc}
\usepackage{graphics, graphicx}
\usepackage{amsmath, amsfonts, amsthm, amssymb}
\usepackage{mathrsfs, mathtools, mathabx}
\usepackage{verbatim} 
\usepackage{enumerate}
\usepackage{xcolor}
\usepackage{cite}
\usepackage{array}
\usepackage[T1]{fontenc}

\newtheorem{theorem}{Theorem}[section]
\newtheorem{definition}[theorem]{Definition}
\newtheorem{lemma}[theorem]{Lemma}
\newtheorem{proposition}[theorem]{Proposition}
\newtheorem{corollary}[theorem]{Corollary}

\newcounter{claimcounter}
\newtheorem{claima}{Claim}[claimcounter]

\newtheorem{claimb}{Claim}[claimcounter]

\DeclareMathOperator{\id}{id}

\newcommand{\A}{\mathcal{A}}

\newcommand{\Z}{\mathbb{Z}} 
\newcommand{\N}{\mathbb{N}}
\newcommand{\Q}{\mathbb{Q}}
\newcommand{\R}{\mathbb{R}}

\newcommand{\h}{\mathfrak{h}}

\newcommand{\Simplex}{\Lambda^\A}
\newcommand{\PermSpace}{S^\A}

\newcommand{\LebR}{\textup{Leb}_\R}

\newcommand{\Function}[5]{\begin{array}{llll} #1 : & #2 & \rightarrow & #3 \\ & #4 & \mapsto & #5 \end{array}}
\date{}
\title	[Antisymmetric extensions of IETs]{On the ergodicity of infinite antisymmetric extensions of symmetric IETs}
\author[P.~Berk]{Przemys\l aw Berk}
\address{Faculty of Mathematics and Computer Science, Nicolaus Copernicus University, ul. Chopina 12/18, 87-100 Toru\'n, Poland}
\email{zimowy@mat.umk.pl}
\author[F.~Trujillo]{Frank Trujillo}
\address{Institut für Mathematik, Universität Zürich, Winterthurerstrasse 190, CH-8057 Zürich, Switzerland}
\email{frank.trujillo@math.uzh.ch }

\begin{document}

\maketitle

\begin{abstract}
We consider skew product extensions over symmetric interval exchange 
transformations {on the unit interval $[0, 1)$} with respect to the cocycle $f 
=\chi_{(0,1/2)}-\chi_{(1/2,1)}$. We prove that for almost every interval 
exchange transformation $T$ with symmetric combinatorial data, the skew product 
$T_f: [0, 1) \times \Z \to [0, 1) \times \Z$ given by $T_f(x,r)=(T(x),r+f(x))$ 
is ergodic with respect to the product of the Lebesgue and counting measures.	
\end{abstract}

\section{Introduction} 

{In this work, we investigate the ergodic properties of a class of} infinite 
extensions of interval exchange transformations. More precisely, given an 
ergodic {(with respect to the Lebesgue measure)} interval exchange 
transformation $T: I \to I$ on the unit interval $I = [0, 1)$ (see Section 
\ref{sc:IETs} for precise definitions) and $f = \chi_{(0,1/2)}-\chi_{(1/2,1)}$, 
we study the ergodicity of the map
\begin{equation}
\label{eq:IET_skewproduct}
\Function{T_f}{[0, 1) \times \Z}{[0, 1) \times \Z}{(x,r)}{(T(x),r+f(x))}
\end{equation}
{with respect to the }product of the Lebesgue and counting measures. 

{
As we shall recall below, general skew products of this kind preserve a natural 
(infinite) measure and exhibit important recurrence properties under the sole 
assumptions of ergodicity of the base transformation $T$ and zero mean of the 
skewing function $f$. However, few results concerning the ergodicity of these 
transformations are known.} The main obstacle {is that this question} falls 
into the class of infinite measure-preserving systems, where the ergodic 
properties are much harder to identify.

\subsection{Recurrence for skew-products} 
{Given a probability measure preserving transformation $(T, X, \mathcal{B}, 
\mu)$ and a measurable \emph{skewing function} $\varphi: X \to G$ (sometimes 
called a \emph{cocycle}) where $(G, +)$ is a locally compact abelian group with 
Haar measure $\nu_G$, it is immediate to check that the associated skew product 
\begin{equation}
\label{eq:general_skewproduct}
\Function{T_\varphi}{X \times G}{X \times G}{(x, g)}{(T(x), g+\varphi(x))}
\end{equation}
preserves the (possibly infinite) product measure $\mu \times \nu_G$. Moreover, 
when 
restricted to $\R$-valued skewing functions, it follows from classical results 
of Atkinson \cite{atkinson_recurrence_1976} and Schmidt 
\cite{schmidt_cocycles_1977} that the associated skew product displays 
important recurrence properties with respect to this invariant measure whenever 
the skewing function has zero mean (see Theorem 
\ref{thm:recurrence_skewproducts} below). {In case of $G=\R$, the measure 
$\nu_{\R}$ is just Lebesgue measure on $\R$, which we denote $\LebR$. On the 
other hand, if $G=\Z$, the measure $\nu_\Z$ is just a counting measure on $\Z$.}

Recall that if $\varphi$ is real-valued, 
$T_\varphi$ is said to be \emph{recurrent} if for every $A \subseteq X$ of 
positive $\mu$-measure and for any $\epsilon > 0$ there exists $n \in \Z$ such 
that $T_\varphi^n(A \times \{0\}) \cap (A \times (-\epsilon, \epsilon)) \neq 
\emptyset$, and $T_\varphi$ is said to be \emph{conservative} if for any $D 
\subseteq X \times \R$ of positive $(\mu \times \LebR)$-measure there exists $n 
\in \Z$ such that $T_\varphi(D) \cap D \neq \emptyset$. 

\begin{theorem}[Atkinson \cite{atkinson_recurrence_1976}, Schmidt \cite{schmidt_cocycles_1977}]
\label{thm:recurrence_skewproducts}
Let $(T, X, \mathcal{B}, \mu)$ be an ergodic measure preserving transformation and $\varphi: X \to \R$ be a measurable function in $L^1(X, \mu)$. Then, the following holds. 
\begin{enumerate}
\item $T_\varphi$ preserves the measure $\mu \times \textup{Leb}_\R$.
\item $T_\varphi$ is recurrent if and only if $\varphi$ has zero mean.
\item $T_\varphi$ is conservative if and only if $\varphi$ has zero mean.
\end{enumerate}
\end{theorem}

The second and third assertions in the theorem above are due to Atkinson and 
Schmidt, respectively. These recurrence properties naturally lead to the 
question of the ergodicity of the skew product, with respect to the product 
measure described above, for skewing functions of zero mean. {However, let us 
point out that if the skewing function $\varphi$ under consideration is 
integer-valued, then the map $T_\varphi$, independently of the map $(T, X, 
\mathcal{B}, \mu)$, is never ergodic with respect to the measure $\mu \times 
\LebR$ (although the conclusions of Theorem \ref{thm:recurrence_skewproducts} 
still hold). Indeed, in this case, the pairs $(x, r) \in X \times \R$ such that 
the fractional part of $r$ lies in the interval $[0,1/2]$ form an invariant set 
of positive (but not full) measure. Nevertheless, we can still consider these 
maps as skew-products over $X \times \Z$ and study their ergodic properties 
with respect to the invariant product measure $\mu \times \nu_\Z$.  Similarly, 
one can naturally extend these considerations to skewing functions taking a 
finite number of values and generating a discrete subgroup of $\R$.}
 
Let us mention that the known methods to prove ergodicity for both $\Z$ and $\R$ valued skewing functions (and, more generally, for cocycles taking values on locally compact abelian groups) are very similar, as they mostly rely on \emph{essential values} (see Section \ref{sc:essential_values}).

\subsection{Previous results on ergodicity}
In the following, we focus our discussion on systems $T_\varphi$ of the form \eqref{eq:general_skewproduct} where the base transformation $T$ is given by an interval exchange transformation (IET) on the unit interval $I$ and the skewing function $\varphi: I \to \R$ has zero mean.} 

For interval exchange transformations with only two intervals, the base system 
$T$ can be naturally identified with a circle rotation. {In this context, the 
family of skew products 
\begin{equation}
\label{eq:rotation_skewproduct}
\Function{S_{\alpha, \beta}}{S^1 \times \R}{S^1 \times \R}{(x,r)}{(x+\alpha\mod 1,\quad r+\varphi_\beta (x))} ,\qquad \varphi_\beta(x)=\chi_{(0,\beta)}-\beta,
\end{equation}
where $(\alpha,\beta)\in S^1 \times S^1$, has been thoroughly studied by several authors. Partial results on the ergodicity of $S_{\alpha, \beta}$, with respect to the product of Lebesgue measures on the circle and the real line, were first obtained by Schmidt \cite{schmidt_cylinder_1978}, for $\alpha = \tfrac{\sqrt{5} - 1}{4},$ $\beta = \tfrac{1}{2}$, and by Conze and Keane \cite{conze_ergodicite_1976}, for any $\alpha \in \R \setminus \Q,$ $\beta = \tfrac{1}{2}$.} This was later generalized by Oren \cite{oren_ergodicity_1983} who gave a complete classification of the cocycles $S_{\alpha, \beta}$ with respect to their ergodic properties. Namely, he proved that for every irrational $\alpha\in S^1$, the map $S_{\alpha, \beta}$ is ergodic if and only if $\beta$ is rational or $1,\alpha$ and $\beta$ are rationally independent, where, for $\beta$ rational, the skew-product is considered on $S^1 \times \frac{1}{q}\Z$ equipped with the product of the Lebesgue and counting measures and $q$ is the smallest possible denominator of $\beta$. 

 In particular, in relation to the main subject of this article, taking $\beta = \frac{1}{2}$ and identifying $\varphi_\beta$ with the map $\frac{1}{2}(\chi_{(0,1/2)}-\chi_{(1/2,1)}) = \frac{1}{2}f$, the results above show that the map $T_f$ as in \eqref{eq:IET_skewproduct} is ergodic, for any minimal IET of two intervals. 
 
 {Other families of ergodic skew products over rotations, considering different 
 classes of skewing functions, have been found by Hellekalek and Larcher 
 \cite{hellekalek_ergodicity_1986}, Pask \cite{pask_skew_1990}, Conze and 
 Piękniewska \cite{conze_multiple_2014}.}

When the number of exchanged intervals is bigger or equal to $3$, there are no 
such complete results. Conze and Frączek \cite{conze_cocycles_2011} have 
constructed a family of ergodic skew products over self-similar interval 
exchange transformations with a cocycle given by a piecewise continuous 
function that is constant over the exchanged intervals. {Ralston and 
Troubetzkoy \cite{ralston_ergodicity_2013} have shown ergodicity for a class of 
odd-valued piecewise constant skew products over a certain two-parameter family 
of IETs. More recently,} Chaika and Robertson \cite{chaika_ergodicity_2019} 
showed that for bounded type interval exchange transformations and typical 
piecewise constant cocycles, the associated skew product is ergodic. {Let us 
point out that the ergodicity of particular classes of skew products over IETs 
has also been established in connection to the ergodicity of certain flows on 
surfaces \cite{ralston_residual_2017, fraczek_ergodicity_2018, 
berk_ergodicity_2023}.

Notice that although skew-products over IETs may be ergodic with respect to the product measures described above, they are generally not uniquely ergodic. For instance, by a result of Aaronson et al. \cite{aaronson_invariant_2002}, for Lebesgue-a.e. $(\alpha,\beta)\in S^1 \times S^1$ the skew product $S_{\alpha, \beta}$ given by \eqref{eq:rotation_skewproduct} admits infinitely many distinct (not proportional) locally finite ergodic measures, known as \emph{Maharam's measures}. 

Finally, let us mention that negative ergodicity results also exist in the literature. Indeed,} Frączek and Ulcigrai \cite{fraczek_non-ergodic_2014} obtained a family of skew products arising from certain infinite billiards, which have very intricate and non-trivial ergodic decomposition.

{
\subsection{An instance of ergodicity over symmetric IETs} In this work,} we 
prove that for almost every \emph{symmetric} interval exchange transformation 
$T$ on the unit interval (see Section \ref{sc:IETs} for a precise definition), 
the skew product $T_f$ with cocycle $f=\chi_{(0,1/2)}-\chi_{(1/2,1)}$ is 
ergodic with respect to the product of Lebesgue and counting measures (see 
Theorem \ref{thm: ergodicity1}). 
According to our knowledge, {this} is the first full-measure result when the 
number of exchanged intervals is bigger than 2. This also partially answers 
Question 1.4 in \cite{chaika_ergodicity_2019}. However, the assumption of the 
symmetricity of the underlying permutation plays a crucial role in our 
argument, and it seems to be impossible to transfer our proof's methods to 
general combinatorial data. Hence the problem of ergodicity of a skew product 
over a typical IET with the cocycle $f=\chi_{(0,1/2)}-\chi_{(1/2,1)}$, 
regardless of the combinatorial data, remains open. 

\section{Main Result}
\label{sc:main}
Our goal is to prove the following. 
{	\begin{theorem}\label{thm: ergodicity1}
		Let $f:[0, 1)\to\Z$ be given by $f=\chi_{(0,1/2)}-\chi_{(1/2,1)}.$
Then, for almost every symmetric interval exchange transformation $T$ on $[0, 1)$, the skew product $T_f: [0,1)\times \Z \to [0,1)\times \Z$ given by \eqref{eq:IET_skewproduct} 
is ergodic with respect to the product of the Lebesgue and counting measures.
	\end{theorem}
}

	Note that the skew products considered above always have {\it ergodic index 1}, that is, the automorphism $T_f\times T_f:\left([0, 1)\times \Z\right)^2\to \left([0, 1)\times \Z\right)^2$ is never ergodic with respect to the square of the product of Lebesgue and counting measures. Indeed, note that for Lebesgue almost every pair $(x,y)\in [0, 1)^2$, we have $f(x)+f(y)=0\mod 
	2$. This implies that the set $[0, 1)^2\times G$, where $G\subset \Z^2$ is the subgroup generated by $(1,1)$ and $(-1,1)$, is a non-trivial $T_f\times T_f$-invariant subset.

Before describing the structure of the proof, let us point out that Theorem \ref{thm: ergodicity1} cannot be extended to all symmetric interval exchange transformations that are ergodic with respect to the Lebesgue measure. Indeed, if $T$ is a symmetric interval exchange transformation such that $\tfrac{1}{2}$ is a left endpoint of one of the exchanged intervals, then 
\[
f = \chi_{(0,1/2)}-\chi_{(1/2,1)}=\chi_{(0,1/2)}-\chi_{(0,1/2)} \circ T,
\]
that is, the function $f$ is a coboundary with respect to $T$, which implies that the \emph{Birkhoff sums} $S_nf(x)$ (see \eqref{eq:Birkhoff_sums} for a definition) are bounded uniformly in $x\in[0,1)$ and $n\in\N$, which prohibits the ergodicity of $T_f$.

{\subsection{Proof strategy} We will prove our main result using 
\emph{essential values} -- a crucial tool used to prove ergodicity. As we shall 
recall in Section \ref{sc:essential_values}, these essential values 
characterize the ergodicity of skew-product under quite general conditions 
(Theorem \ref{thm: erg_criterion}). Moreover, for systems such as 
\eqref{eq:IET_skewproduct}, we can provide sufficient conditions for ergodicity 
(Proposition \ref{prop:intessentialvalue}) through \emph{partial rigidity 
sequences} (see Section \ref{sc:essential_values} for precise definitions of 
these objects). We will build these sequences of sets by exploiting the 
symmetricity assumption on the IET to control Birkhoff sums of the form 
$S_nf(\tfrac{1}{2})$ (see Corollary \ref{cor: cancellations}) and extending 
this control to larger sets through control on the location of $\tfrac{1}{2}$ 
in the Rohlin towers associated to the \emph{Rauzy-Veech inductions} of $T$ 
(see Sections \ref{sc:IETs} and \ref{sc:centers}). This type of control on the 
location was previously shown in \cite{berk_ergodicity_2023-1}, and the precise 
formulation we give here (see Lemma \ref{lem: positions_of_centers}) is a 
direct consequence of the results therein.

}
{As mentioned before, our methods do not enable us to obtain similar results 
for permutations other than symmetric ones. Indeed, to show the existence of 
non-trivial essential values, one of the most crucial elements is to find 
arbitrarily large times $n \in \N$ for which the Birkhoff sums $S_n f$ are 
uniformly bounded on a relatively large set. For this purpose, in Section 
\ref{sec: symmetricinv}, we use the fact that every symmetric IET is conjugate 
to its inverse via symmetric reflection of the interval. This property is 
particular to symmetric IETs, and we are unaware of other methods that could 
allow us to extend our result to the general situation.}

{
\subsection{Outline of the article} The remainder of the paper is devoted to 
proving Theorem \ref{thm: ergodicity1}. Section \ref{sc:IETs} contains the 
necessary background on interval exchange transformations and introduces the 
notations used throughout the paper. In Section \ref{sc:essential_values}, we 
recall the notion of essential values and some additional facts and criteria 
concerning the ergodicity of skew products, taken mainly from 
\cite{conze_cocycles_2011}. Next, in Section \ref{sec: symmetricinv}, we 
present some facts concerning symmetric interval exchange transformations and 
\emph{odd cocycles} (see Definition \ref{def: odd_cocycle}), which will play a 
crucial role in the proof of our main result. This section should be compared 
with some parts of \cite{berk_ergodicity_2023-1}, where the idea of exploiting 
the symmetric permutation assumption is introduced to deal with a particular 
class of odd cocycles. {It should be highlighted that although the definition 
of odd cocycles used in \cite{berk_ergodicity_2023-1} is different from that 
given in the present work, the effects of these assumptions are similar.} 
Section \ref{sc:centers} introduces a result concerning the location of 
$\tfrac{1}{2}$ in the Rohlin towers associated with the Rauzy-Veech inductions 
of a typical IET $T$, which we borrow from \cite{berk_ergodicity_2023-1}. 
Finally, we prove our main result in Section \ref{sec: char}. }

\section{Interval exchange transformations}
\label{sc:IETs}
{We now present the objects used throughout the article. We derive our notation 
mostly from \cite{berk_backward_2022}.} Recall that an \emph{interval exchange 
transformation (IET)} $T$ of a bounded interval $I$ is a bijection on $I$ that 
is a piecewise translation. Such a map is fully characterized by the number of 
exchanged intervals, their length, and the ordering before and after 
transformation. More precisely, we can encode an IET with the help of a 
\emph{finite alphabet} $\mathcal A$ of $d\ge 2$ elements, which we use to index 
the \emph{exchanged intervals} $\{I_\alpha\}_{\alpha \in \A}$, a \emph{lengths 
vector} $\lambda \in \R_+^\A$ describing the length of these intervals, that is,
\[ \lambda_{\alpha}=|I_{\alpha}|, \quad \text{for every }\alpha\in\A,\]
where $|J|$ denotes the length of any subinterval $J \subseteq I$, and a \emph{permutation} $\pi$ encoding the exchange order. As the intervals are indexed using the finite alphabet $\A$, we consider $\pi$ as a pair $\pi=(\pi_0,\pi_1)$, where $\pi_0,\pi_1:\mathcal A\to \{1,\ldots,d\}$ are bijections describing the order of the intervals before and after the exchange, respectively.

Throughout this work, when considering an IET with permutation $\pi$, we will always assume that the permutation is \emph{irreducible}, i.e., that there is no $1\le k<d$ such that $\pi_1\circ \pi_0^{-1}\left(\{1,\ldots,k\}\right)=\{1,\ldots,k\}$. Otherwise, the system splits into two subsystems, both being interval exchange transformations, which can be considered separately. We denote the set of all irreducible permutations on a finite alphabet $\A$ by $\PermSpace$. 

In this article, we will be mostly concerned with \emph{symmetric} permutations, that is, permutations $\pi = (\pi_0, \pi_1) \in \PermSpace$ satisfying
\[
\pi_1\circ\pi_0^{-1}(k)=d+1-k,\quad\text{for every }k\in\{1,\ldots,d\}.
\]
We refer to the IETs with symmetric combinatorial data as \emph{symmetric IETs}. 

Normally, the lengths vector $\lambda$ of a given IET is an element of $\R_+^{\A}$. However, it is sometimes useful to consider the set of normalized lengths, that is, when 
$$|\lambda| = \sum_{\alpha \in \A} \lambda_\alpha = |I|=1.$$ 
In this case, $\lambda$ is an element of the real positive unit simplex 
$$\Simplex:=\left\{\lambda\in \R_{+}^{\mathcal A} \,\left|\,|\lambda|=1 \right\}\right. .$$
From now on, if there is no risk of confusion, we will often identify an IET $T$ with a pair $(\pi, \lambda) \in \PermSpace \times \R_+^\A$ as defined above. We endow $\R^{\A}_{+}$ and $\Simplex$ with the Lebesgue measure. 

An important tool to investigate the properties of a typical IET (with respect 
to the product of the counting measure on $\PermSpace$ and the Lebesgue measure 
on $\R^{\A}_{+}$) is the \emph{Rauzy-Veech induction}, which we denote by $R$. 
This induction procedure assigns to a.e. IET $(\pi,\lambda)\in \PermSpace\times 
\R^{\A}_{+}$ with exchanged intervals {$\{I_\alpha\}_{\alpha \in \A}$}, another 
interval exchange transformation $R(\pi, \lambda) \in \PermSpace\times 
\R^{\A}_{+}$ (in particular, this new IET has the same number of exchanged 
intervals as the initial one) by considering the first return map to either the 
interval $I\setminus I_{\pi_0^{-1}(d)}$ or $I\setminus I_{\pi_1^{-1}(d)}$, 
whichever is longer. If these two intervals have the same length, then $R$ is 
not well-defined. However, this is an exceptional scenario as for a.e. IET 
$(\pi,\lambda)\in \PermSpace\times \R^{\A}_{+}$, the induction $R$ can be 
iterated indefinitely. For the sake of simplicity, we will assume that $R$ is 
defined for every IET. 

Given $T=(\pi,\lambda) \in \PermSpace\times \R^{\A}_{+}$, we denote $T_n=R^n(\pi,\lambda)=(\pi^n,\lambda^n)$, for any $n\in\N$. The domain of $T_n$ is denoted by $I^n$ and the intervals exchanged by $T_n$ are denoted by $\{I^n_{\alpha}\}_{\alpha\in\A}$.

{One of the main reasons to {use} the Rauzy-Veech induction $R$ is to construct 
Rohlin towers for a given IET $T = (\pi, \lambda)$. Namely, for every $n \in 
\N$, we consider the intervals $I^n_{\alpha}$ and their \emph{return times} to 
$I^n$ via $T$, which we denote by $h^n_\alpha\in\N$. In this way, we obtain $d$ 
distinct Rohlin towers of the form $\{I^n_{\alpha}, T(I^n_{\alpha}),\dots, 
T^{h^n_\alpha - 1}(I^n_{\alpha})\}$, with $\alpha \in \A$, to which we refer as 
\emph{Rauzy-Veech towers}. It is not difficult to see that $T$ acts on each 
tower level via translation.}

The {minimal nonempty} sets of irreducible permutations $\mathfrak{R} \subseteq 
\PermSpace$ such that $\mathfrak{R} \times \R^\A_+$ is invariant under the 
action of $R$ are called \emph{Rauzy classes}. The sequence of permutations 
associated with the IET $R^n(\pi,\lambda)$ is called a \emph{Rauzy orbit}. Any 
finite subsegment $\gamma$ of the Rauzy orbit is called {a} \emph{Rauzy path}, 
and we denote its length by $|\gamma|$. A Rauzy path $\gamma$ is said to be 
{\emph{positive} if for every $(\pi,\lambda)$ that follows the path $\gamma$ by 
Rauzy-Veech induction and every $\alpha,\beta\in\A$, there exists $0\le k 
<h^{|\gamma|}_{\alpha}$ such that $T^k(I^{|\gamma|}_\alpha)\subset 
I^{|\gamma|}_\beta$. {By \cite[Lemma 1.2.4]{marmi_cohomological_2005}, for any 
IET for which the Rauzy-Veech induction can be iterated indefinitely, {there 
exists $n \in \N$ such that the Rauzy path obtained after $n$ iterations of the 
Rauzy-Veech induction is positive.}}}


One can also consider a \emph{normalized Rauzy-Veech induction} $\tilde R: \PermSpace\times \Simplex\to \PermSpace\times \Simplex$, which assigns to every IET $(\pi,\lambda)\in \PermSpace\times\Simplex$ the IET $R(\pi,\lambda)$ linearly rescaled by the length of the domain. The importance of this object follows from the following well-known results obtained independently by Masur \cite{masur_interval_1982} and Veech \cite{veech_gauss_1982}.
\begin{theorem}[]\label{thm: RVergodic}
For every Rauzy {class} $\mathfrak{R} \subseteq \PermSpace$, the restriction of 
the normalized Rauzy-Veech induction $\tilde{R}$ to $\mathfrak{R} \times 
\Simplex$ is ergodic with respect to an invariant {infinite} measure. Moreover, 
this measure is equivalent to the product of the counting measure on 
$\mathfrak{R} $ and the Lebesgue measure on $\Simplex$. 
\end{theorem}

Let us recall the following well-known fact concerning the return times of the 
Rauzy-Veech towers and {positive} Rauzy paths.

\begin{proposition}[]\label{prop: gammaconstant}
Let $\gamma$ be a finite {positive} Rauzy path in $\PermSpace$. There exists 
$\rho(\gamma) > 0$ such that for every IET $(\pi,\lambda)\in 
\PermSpace\times\R_{+}^\A$ and every $n\ge|\gamma|$, if the Rauzy path obtained 
by acting $|\gamma|$ times on $R^{n-|\gamma|}(\pi,\lambda)$ by $R$ is equal to 
$\gamma$, then
\[
\max_{\alpha,\beta\in\A}\frac{h^n_{\alpha}}{h^n_{\beta}}<\rho(\gamma).
\] 
\end{proposition}
{This is a direct consequence of Proposition 4.23 in \cite{viana_ergodic_2006}. 
Indeed, one can check that the fact of $\gamma$ being positive is equivalent to 
the so-called \emph{Rauzy matrix} {associated to $\gamma$} being positive (we 
refer the reader to \cite{viana_ergodic_2006} for a precise definition).} The 
most important aspect of the above result is the independence of the constant 
$\rho(\gamma)$ from the exact parameters of the IET.

A typical strategy (which we also follow) to control the lengths of the bases 
and the heights of the Rauzy-Veech towers is to fix a permutation $\pi$ (in our 
case, a symmetric one), choose an appropriate positive Lebesgue measure set $A 
\subseteq \Simplex$ (which corresponds to the possible lengths of the bases) 
such that for every $\lambda \in A$ the Rauzy orbit of $(\pi, \lambda)$ ends 
with the same {positive} Rauzy path $\gamma$, and use the ergodicity of $\tilde 
R$ to obtain the needed proportions of lengths of the bases and heights of the 
Rauzy-Veech towers for the iterates by $\tilde R$ that belong to the set 
$\{\pi\} \times A$. Notice that in this case, the control in the proportions of 
the heights follows from Proposition \ref{prop: gammaconstant}.

For additional details on interval exchange transformations and Rauzy-Veech induction, we refer the reader to \cite{viana_ergodic_2006} and \cite{yoccoz_interval_2010}.

\section{Skew products and essential values}
\label{sc:essential_values}

We intend to investigate the ergodicity of skew products through \emph{essential values}, first introduced by Schmidt in \cite{schmidt_cocycles_1977}. 

{Let $(T, X, \mathcal{B}, \mu)$ be a probability measure preserving 
transformation and $\varphi: X \to G$ be a measurable cocycle, where $(G, +)$ 
is a locally compact abelian group with Haar measure $\nu_G$. Let $\| 
\cdot\|_G$ be a topology-generating (group) norm for $G$ (see \cite[Corollary 
1.6.10]{aaronson_introduction_1997}). We say that $a \in G$ is an 
\emph{essential value} of $T_\varphi$ iff for every $A\in\mathcal B$ such that 
$\mu(A)>0$ and every $\epsilon > 0$ there exists $n\in\N$ satisfying
\[
\mu\left\{x\in A \,\left|\, T_\varphi^n(x,0) \in A \times \{g \in G \mid \|g - a\|_G < \epsilon \} \right.\right\}>0.
\]
We denote the set of all essential values of $T_\varphi$ by $Ess(T_\varphi)$. 

The following result shows that we can characterize ergodicity in terms of essential values. For a proof and additional details, we refer the interested reader to \cite[\S8.2]{aaronson_introduction_1997}.

\begin{theorem}[Essential values criterion]
\label{thm: erg_criterion}
	Assume that $(T, X,\mathcal B,\mu)$ is an ergodic probability measure preserving transformation and $\varphi: X \to G$ is measurable, where $G$ is a locally compact abelian group with Haar measure $\nu_G$. Then either $Ess(T_\varphi) = \emptyset$ or $Ess(T_\varphi)$ is a closed subgroup of $G$. Moreover, the skew product $T_\varphi$ is ergodic with respect to $\mu \times \nu_G$ if and only if $Ess(T_\varphi)=G$. 
\end{theorem}

Using the following definition, we will provide sufficient conditions for an integer to be an essential value of a skew product over a compact metric space with a $\Z$-valued cocycle.} Given a probability measure preserving transformation $(T, X,\mathcal B,\mu)$ on a compact metric space $X$, a sequence of measurable subsets $(\Xi_n)_{n \in \N}$ of $X$ is called a \emph{partial rigidity sequence} if there exists an increasing sequence of natural numbers $(\h_n)_{n \in \N}$, called \emph{rigidity times}, and $\delta>0$ such that: 
\begin{enumerate}[(i)]
	\item \label{cond:towers_measure} $\mu(\Xi_n) \to \delta > 0,$ 
	\item \label{cond:quasi_invariance} $\mu(\Xi_n \Delta T(\Xi_n)) \to 0$, 
	\item \label{cond:partial_rigidity} $\sup_{x \in \Xi_n} d(x, T^{\h_n}x) \to 0$.
\end{enumerate}
In this case, we say that $(\Xi_n)_{n \in \N}$ is a sequence of \emph{$\delta$-partial rigidity sets} along the sequence $(\h_n)_{n \in \N}$ of rigidity times. 

Restricted to integer-valued skewing functions, {one can provide sufficient 
conditions for a given integer number to be an essential value of a skew 
product in terms of partial rigidity sequences and Birkhoff sums  (see 
\eqref{eq:Birkhoff_sums} below for the definition) as follows.}

\begin{proposition}[Essential values for integer-valued cocycles, Corollary 2.8 
in \cite{conze_cocycles_2011}] 
	\label{prop:intessentialvalue}
	Let $(T, X, \mathcal{B}, \mu)$ be an ergodic probability measure preserving 
	transformation of a compact metric space $X$, and let $\varphi: X \to \Z$ 
	be an integer-valued cocycle in $L^1(X, \mu)$ with zero mean. Suppose there 
	exist partial rigidity sets $(\Xi_n)_{n \in \N}$ and corresponding rigidity 
	times $(\mathfrak{h}_n)_{n \in \N}$ such that $S_{\mathfrak{h}_n} 
	\varphi|_{\Xi_n}=a$ (see \eqref{eq:Birkhoff_sums} below), for some 
	$a\in\Z$. Then $a \in Ess(T_\varphi)$. 
\end{proposition}

	\section{Symmetric interval exchange transformations and involutions}
	\label{sec: symmetricinv}

Recall that, for any IET $T = (\pi,\lambda)$, its inverse $T^{-1}$ is also an IET on the same number of intervals and can be represented as $T^{-1} = (\overline{\pi}, \lambda)$, where
 \[ \overline \pi = (\overline \pi _0, \overline \pi_1) = (\pi_1, \pi_0).\]
That is, $T^{-1}$ exchanges the intervals $\{T(I_\alpha)\}_{\alpha \in \A}$ according to $\overline{\pi}.$ In particular, if $T$ is a symmetric IET, then $T^{-1}$ is also a symmetric IET. 

Let us denote by $\mathcal I$ the involution of the interval $[0, 1]$ given by
\[\Function{\mathcal I}{[0,1]}{[0,1]}{x}{1-x}.\]
 It is easy to check that if $T = (\pi,\lambda) \in \PermSpace \times \Simplex$ is a symmetric IET on $[0, 1]$, then 
 \begin{equation}\label{eq: IET_and_involution}
	T^{-n}\circ\mathcal I=\mathcal I\circ T^n, 
	\end{equation}
	for every $n \in \N$. 	Indeed, for $n = 1$, if $x \in I_\alpha$ then
		\[T(x) = x - \sum_{\pi_0(\beta)<\pi_0(\alpha) } \lambda_\beta + \sum_{\pi_1(\beta) < \pi_1(\alpha) }\lambda_\beta.\]
Notice that, since $T$ is symmetric, $1 - x \in T(I_\alpha) = \mathcal I (I_\alpha)$. Hence
\begin{align*}
T^{-1}(1 - x) & = 1 - x - \sum_{\overline \pi_0(\beta)< \overline \pi_0(\alpha) } \lambda_\beta + \sum_{\overline \pi_1(\beta) < \overline \pi_1(\alpha) }\lambda_\beta \\
& = 1 - x - \sum_{\pi_1(\beta)<\pi_1(\alpha) } \lambda_\beta + \sum_{\pi_0(\beta) < \pi_0(\alpha) }\lambda_\beta \\ &= 1 - T(x).
\end{align*}
Thus $T^{-1} \circ \mathcal I = \mathcal I \circ T$. For $n \geq 1$, property \eqref{eq: IET_and_involution} now follows easily by induction. 	Since $\mathcal I$ is an involution, i.e. $\mathcal I^2 = \id,$ it follows that $\mathcal I\circ T^{-1} =T\circ\mathcal I,$ and then, again by induction, property \eqref{eq: IET_and_involution} holds for every $n\le -1$.
	
For every $n\in\Z$, we denote the \emph{Birkhoff sums} of $\varphi: [0, 1) \to 
\R$ with respect to $T$ by 
\begin{equation}
\label{eq:Birkhoff_sums}
S_n \varphi(x):=\begin{cases}
\sum_{i=0}^{n-1}\varphi(T^{i}x)& \text{ if }n\ge 1,\\
0&\text{ if }n=0,\\
-\sum_{i=-n}^{-1}\varphi(T^{i}x)&\text{ if }n\le 1,
\end{cases}
\end{equation}
where, for the sake of simplicity, we omit the dependence on $T$ in the 
notation. 

The following assumption on cocycles will be crucial to obtain estimates on Birkhoff sums with respect to a symmetric IET (see Lemma \ref{lem: odd_values}). 
\begin{definition}
\label{def: odd_cocycle}
	We say that a function $\varphi:[0,1)\to\R$ is odd if for every $x\in(0,1)$ we have 
	\[
	\varphi(x) = -\varphi\circ \mathcal I(x).
	\]
\end{definition}
\noindent Note that the function $f=\chi_{(0,1/2)}-\chi_{(1/2,1)}$ is odd.

\begin{lemma}
\label{lem: odd_values}
	Assume that $T$ is a symmetric IET and $\varphi:[0,1)\to\R$ is odd. Then, for every $n\ge 0$, we have
	\[
	S_{n+1}\varphi(\tfrac{1}{2})=S_{-n}\varphi(\tfrac{1}{2}).
	\]

\end{lemma}
\begin{proof}
	Let $n\ge 0$. Since $\varphi$ is odd, we have $\varphi(\tfrac{1}{2})=0$. Moreover, $\mathcal I(\tfrac{1}{2})=\tfrac{1}{2}$. Thus,
	\[
	S_{n+1}\varphi(\tfrac{1}{2})=\sum_{i=0}^{n}\varphi\big(T^{i}(\tfrac{1}{2})\big)=\sum_{i=1}^{n} \varphi \big(T^{i}(\tfrac{1}{2})\big)=\sum_{i=1}^{n} \varphi \big(T^{i}\circ \mathcal I(\tfrac{1}{2})\big).
	\]
	Then, by \eqref{eq: IET_and_involution}, we get
	\[
	\sum_{i=1}^{n}\varphi\big(T^{i}\circ \mathcal I(\tfrac{1}{2})\big)=\sum_{i=-n}^{-1}\varphi\big(\mathcal I\circ T^i(\tfrac{1}{2})\big)=-\sum_{i=-n}^{-1}\varphi(T^i(\tfrac{1}{2}))=S_{-n}\varphi(\tfrac{1}{2}),
	\]
	which concludes the proof.
	
\end{proof}
	\begin{corollary}\label{cor: cancellations}
		Assume that $T$ is a symmetric IET and $\varphi: [0,1)\to\R$ is odd. Then, for every $n\ge 0$, we have
\[
S_{2n}\varphi(T^{-n}(\tfrac{1}{2}))= - \varphi(T^{n}(\tfrac{1}{2})).
\]

	\end{corollary}
\begin{proof}
This follows from Lemma \ref{lem: odd_values}. Indeed, note that
\[
S_{-n}\varphi(x)=-S_{n}\varphi(T^{-n}(x)), \qquad 
S_{2n}\varphi(x)=S_{n}\varphi(x)+S_n\varphi(T^n(x))
\]
and thus
\[
S_{2n}\varphi(T^{-n}(x))=S_{n}\varphi(x)-S_{-n}\varphi(x)
\]
for any $x \in [0, 1)$ and any $n \in \Z$. Hence, by Lemma \ref{lem: 
odd_values}, we get
\[
S_{2n}\varphi(T^{-n}(\tfrac{1}{2}))=S_{n}\varphi(\tfrac{1}{2})-S_{n+1}\varphi(\tfrac{1}{2})
=-\varphi(T^n (\tfrac{1}{2})).
\]
\end{proof}

\section{Rauzy-Veech induction and the position of centers}
\label{sc:centers}
Given an IET $T$ exchanging the intervals $\{I_\alpha\}_{\alpha \in \A}$, we denote by $c_{\alpha}\in I_\alpha$ the center of the interval $I_{\alpha}$, for every $\alpha\in\A$. Analogously, for every $n\in\N$, we denote by $c_{\alpha}^n$ the center of the interval $I_{\alpha}^n$, for every $\alpha \in \A$. For any $x\in I$ and any $n\in\N$, we denote by $p_n(x)$ the \emph{projection} of $x$ on $I^n$, that is, $p_n(x)=T^{-l}(x)$, where $l\ge 0$ is the smallest natural number such that $T^{-l}(x)\in I^n$.

{The following lemma gives crucial information about the positions of centers 
of intervals, {as well as $\tfrac{1}{2}$,} inside the Rauzy-Veech towers.
\begin{lemma}\label{lem: positions_of_centers}
For almost every symmetric IET $T=(\pi,\lambda)$ and for every subset 
$A\subseteq \Simplex$ of positive Lebesgue measure, there exists a finite 
positive Rauzy path $\gamma$, and an increasing sequence $\{n_k\}_{k \in \N} 
\subseteq \N$ such that, for any $k \in \N$, $T_{n_k}$ is symmetric, 
$\frac{\lambda^{n_k}}{|\lambda^{n_k}|}\in A$, {the Rauzy path associated to the 
last $|\gamma|$ Rauzy-Veech induction steps before $T_{n_k}$ is given by 
$\gamma$}, and 
\begin{equation*}\label{eq: centerpositions}
\{{p_{n_k}(c_\alpha) } \mid \alpha\in\mathcal 
A\}\cup\{p_{n_k}(\tfrac{1}{2})\}=\{c^{n_k}_{\alpha} \mid 
\alpha\in\A\}\cup\{{|I^{n_k}|/2}\}.
\end{equation*}
\end{lemma}

{The above result follows from \cite[Lemma 3.12]{berk_ergodicity_2023-1}. To 
obtain Lemma \ref{lem: positions_of_centers}, it is enough to replace the set 
$U$ by the set $A$ ({from the assumptions of} Lemma \ref{lem: 
positions_of_centers}) in the proof of Lemma 3.12 in 
\cite{berk_ergodicity_2023-1}.} }

Lemma \ref{lem: positions_of_centers} is one of the crucial steps in locating $\tfrac{1}{2}$ in the Rauzy-Veech towers. More precisely, for any $k \geq 1$, this lemma shows that $\tfrac{1}{2}$ is in the forward orbit of either the middle-point of the domain $I^{n_k}$ or the middle-point of one of the intervals exchanged by $T_{n_k}$. Moreover, in the latter case, it is clear that $\tfrac{1}{2}$ is the middle point of the level of the Rauzy-Veech tower to which it belongs. To obtain a similar control in the case where $p_{n_k}(\tfrac{1}{2})=\frac{1}{2}|I^{n_k}|$ it is sufficient to assume that $\frac{1}{2}|I^{n_k}|$ is sufficiently far from any of the discontinuities of $T_{n_k}$. 

{As we shall see in the next section, under the assumptions of Theorem 
\ref{thm: ergodicity1} and after suitably choosing a set $A \subseteq 
\Simplex$, appropriate control of the location of $\tfrac{1}{2}$ (see Claim 
\ref{claim:pos_1/2}) in the Rauzy-Veech towers associated to the 
renormalization times given by Lemma \ref{lem: positions_of_centers}, together 
with Corollary \ref{cor: cancellations}, will allow us to construct rigidity 
sets as in Proposition \ref{prop:intessentialvalue}.

}

\section{Proof of Theorem \ref{thm: ergodicity1}}\label{sec: char}
In this section, we prove our main result by showing that $-1$ is an essential 
value of $T_f$. {Recall that by Theorems \ref{thm: erg_criterion} and 
Proposition \ref{prop:intessentialvalue} this is enough to guarantee ergodicity 
of the skew-product. }{For this purpose, we use the jump discontinuity of $f$ 
to construct rigidity sets as in Proposition \ref{prop:intessentialvalue}. 
Indeed, {using Corollary \ref{cor: cancellations} and Lemma \ref{lem: 
positions_of_centers} (after properly choosing the set $A$), we will prove that 
for a.e. symmetric IET $T$, there exists an increasing sequence 
$(\mathfrak{h}_k)_{k \in \N}$ such that $S_{\mathfrak{h}_k}f(\tfrac{1}{2}) = 
0$. Hence, for $k \geq 1$ fixed and assuming that $\tfrac{1}{2}$ does not lie 
in the orbit of any of the discontinuities of $T$, if $x$ is sufficiently close 
to $\tfrac{1}{2}$ from the left (resp. right) we have $S_{\mathfrak{h}_k}f(x) = 
1$ (resp. $S_{\mathfrak{h}_k}f(x) = -1$). Using Lemma \ref{lem: 
positions_of_centers} to locate $\tfrac{1}{2}$, we will show that we can extend 
these estimates, say $S_{\mathfrak{h}_k}f(x) = -1$, to a partial rigidity 
sequence of sets $(\Xi_k)_{k \in \N}$ with rigidity times given by 
$(\mathfrak{h}_k)_{k \in \N}$.}} 

{The sequence $\mathfrak{h}_k$ described above will be chosen as the sum of a 
finite number of heights of Rohlin towers at the $n_k$ step of Rauzy-Veech 
induction, where $(n_k)_{k \in \N}$ is the sequence given by Lemma \ref{lem: 
positions_of_centers}. Moreover, as we shall see in the proof, to prove the 
properties stated above, it will be crucial that the points 
$T^{-\mathfrak{h}_k}(\tfrac{1}{2}), T^{\mathfrak{h}_k}(\tfrac{1}{2})$ belong to 
the same floor and same Rohlin tower as $\tfrac{1}{2}$, and that these three 
points remain relatively close to the middle point of the floor to which they 
belong. As we shall explain below, since $T$ is a symmetric IET, this can be 
easily achieved if the exchanged intervals have {comparable} length. 


 
 }
{With this in mind, we define $A$ as}
\begin{equation}
\label{eq:def_A_odd} 
A=\left\{\lambda\in\Simplex\ \left|\ 
\tfrac{1}{{d}}-\tfrac{1}{100{d}^3}<\lambda_{\alpha}<\tfrac{1}{{d}}+\tfrac{1}{100{d}^3},
 \quad\forall \alpha\in \A\right\}\right.,
\end{equation}
if ${d}$ is odd, or
\begin{equation}
\label{eq:def_A_even} 
A = \left\{ \lambda\in\Simplex \,\left|\, \begin{array}{l} 
\frac{1}{{d} - 
1}-\frac{1}{100{d}^3}<\lambda_{\alpha}<\frac{1}{{d} - 
1}+\frac{1}{100{d}^3},\quad \forall \alpha \in \A 
\setminus\{\pi_0^{-1}(1), \pi_0^{-1}(2)\}; \\ 
\frac{3}{7}(\frac{1}{{d} - 
1}-\frac{1}{100{d}^3})<\lambda_{\pi_0^{-1}(1)}<\frac{3}{7}(\frac{1}{{d}
 - 1}+\frac{1}{100{d}^3}); \\ \frac{4}{7}(\frac{1}{{d} - 
1}-\frac{1}{100{d}^3})<\lambda_{\pi_0^{-1}(2)}<\frac{4}{7}(\frac{1}{{d}
 - 1}+\frac{1}{100{d}^3}).
\end{array}\right\}\right.,
\end{equation}
if ${d}$ is even. The reason why we differentiate between these two cases is 
that Lemma \ref{lem: positions_of_centers} does not specify precisely the 
position of $\tfrac{1}{2}$ in the Rauzy-Veech towers but provides at most $d + 
1$ different possibilities. {Hence, it {is} possible that along the subsequence 
$(n_k)_{k \in \N}$ given by Lemma \ref{lem: positions_of_centers}, 
$\tfrac{1}{2}$ always lies in the forward orbit of the middle point of 
$I^{n_k}$.} If we picked the set $A$ in the even case exactly as we do in the 
odd case, then, in the previous scenario, $\tfrac{1}{2}$ would be close to a 
discontinuity of $T_{n_k}$ and thus close to the boundary of the floor of the 
Rohlin tower to which it belongs. {As mentioned before, we want to avoid this 
from happening to extend the estimate on Birkhoff sums starting at 
$\tfrac{1}{2}$ to large partial rigidity sequences. 
}
We fix this issue by imposing a slightly less balanced condition in the lengths of the set $A$ in the even case. The constants $\frac{3}{7}$ and $\frac{4}{7}$ are chosen so that sufficiently long orbits of the centers of exchanged intervals do not come too close to the discontinuities of $T_{n_k}$ either. Apart from this slight technical difference, the proofs are essentially the same in both cases.	

{Let us further motivate the definition of this set by considering the 
particular case where $\lambda_\alpha = \tfrac{1}{d}$ for all $\alpha \in \A$, 
if $d$ is odd, and $\lambda_{\pi_0^{-1}(1)} = \tfrac{3}{7(d - 1)}$, 
$\lambda_{\pi_0^{-1}(2)} = \tfrac{4}{7(d - 1)}$, $\lambda_\alpha = \tfrac{1}{d 
- 1}$, for $\alpha \in \A \setminus\{\pi_0^{-1}(1), \pi_0^{-1}(2)\}$, if $d$ is 
even. Notice that the set $A$ is nothing more than an open set consisting of 
small perturbations of these parameters. Let us consider a symmetric IET $T_0$ 
with lengths given by $\lambda$. 
	In this case, $\tfrac{1}{2}$ is a fixed point of $T_0$ and coincides with 
	$c_{\beta}$, where $\beta = \pi_0^{-1}(\lfloor\tfrac{d + 1}{2}\rfloor)$. If 
	$d$ is odd, the centers $c_\alpha$, for $\alpha \in \A$, are 2-periodic 
	points for $T_0$. On the other hand, if $d$ is even, $c_{\pi_0^{-1}(1)}$, 
	$c_{\pi_0^{-1}(2)}$ and $c_{{\pi_0^{-1}(d)}}$ are no longer 2-periodic 
	points while the other centers remain 2-periodic. 
	However, for $\alpha \in\{ \pi_0^{-1}(1), \pi_0^{-1}(2),{\pi_0^{-1}(d)}\}$, 
	the iterates $T_0^4(c_\alpha)$ and $T_0^{-4}(c_\alpha)$ are at a distance $\tfrac{1}{7(d - 1)}$ 
	of $c_\alpha$, and thus belong to $I_\alpha$ and {stay far} from its 
	boundary. 
	Therefore, given an IET $T$ as in Lemma \ref{lem: positions_of_centers} 
	such that its $n_k$-th Rauzy-Veech renormalization $T_{n_k}$ is close to 
	$T_0$, if we define $\mathfrak{h}_k$ as the sum of either two or four 
	heights of the associated Rohlin towers (depending on the tower to which 
	$\tfrac{1}{2}$ belongs), we can guarantee, by Lemma \ref{lem: 
	positions_of_centers}, that $T^{-\mathfrak{h}_k}(\tfrac{1}{2})$ and 
	$T^{\mathfrak{h}_k}(\tfrac{1}{2})$ {are on }the same floor and same Rohlin 
	tower as $\tfrac{1}{2}$ and that they are all relatively close to the 
	middle of the floor. It is worth adding that in the even case, the points $T^{-4}_0\big(c_{\pi_0^{-1}(d)}\big)$ and $c_{\pi_0^{-1}(d)}$ travel through the same exchanged intervals under first four iterates of $T_0$, although not in the same order. However, the fact that the total number of each of the intervals visited is the same, as in this case, suffices to compare the orbits of those two points.


}

 Define $A$ as in \eqref{eq:def_A_odd} or \eqref{eq:def_A_even}, depending on whether $d$ is odd or even. Let $(n_k)_{k \in \N}$ and $\gamma$ be given by Lemma \ref{lem: positions_of_centers}. We will use Proposition \ref{prop:intessentialvalue} to show that $-1$ is an essential value of $T_f$, which implies, by Theorem \ref{thm: erg_criterion}, that $T_f$ is ergodic. 
To apply Proposition \ref{prop:intessentialvalue}, we need to define appropriate partial rigidity sets together with rigidity times. For this, it will be crucial to control the position of $\tfrac{1}{2}$ in the Rauzy-Veech towers associated with the sequence of renormalizations given by Lemma \ref{lem: positions_of_centers} and to find rigidity times for appropriate subsets of these towers. Indeed, finding appropriate sub-towers whose iterates, up to a certain rigidity time, avoid $\tfrac{1}{2}$ will allow us to construct rigidity sets on which the Birkhoff sums at the rigidity times are constant.

 In view of Lemma \ref{lem: positions_of_centers}, there exists $\alpha\in\mathcal A$ such that, passing to a subsequence if necessary, $\tfrac{1}{2} \in T^{\ell_k}(I^{n_k}_{\alpha})$ for some $0\le\ell_k< h^{n_k}_{\alpha},$ for every $k\in\N$. 
 In the following, we denote $I_\alpha^n = [l_\alpha^n, r_\alpha^n).$ Recall that we denote by $c^n_\alpha = \tfrac{1}{2}(l^n_\alpha + r^n_\alpha)$ the middle point of the interval $I^n_\alpha$, for every $n \in \N$. 
 
 We will first prove Theorem \ref{thm: ergodicity1} under the assumption that 
 ${d}$ is odd, {which we refer to as Case A. The proof in the even case, which 
 we refer to as Case B, follows a similar approach and will be outlined towards 
 the end of the proof.

 }
 \medskip
\noindent \textbf{Case A: ${d}$ is odd.} We illustrate this situation in Figure 
\ref{fig:odd}. {Let us first investigate the relations between lengths of 
intervals obtained via Rauzy-Veech induction. 
\stepcounter{claimcounter}\begin{claima}\label{claim:pos_1/2}
For every $k\in\N$, we have
	\begin{equation*}
		\max_{\beta,\beta'\in\mathcal 
		A}\left|\lambda_{\beta}^{n_k}-\lambda_{\beta'}^{n_k}\right|<\frac{\lambda_{\alpha}^{n_k}}{{40}{d}^2}.
	\end{equation*}
\end{claima}

		\begin{proof}[Proof of Claim \ref{claim:pos_1/2}]
	By the definition of $A$, we obtain
	\[
	\max_{\beta,\beta'\in\mathcal 
	A}\left|\lambda_{\beta}^{n_k}-\lambda_{\beta'}^{n_k}\right|<\frac{1}{50{d}^3}<\frac{100{d}\lambda_{\alpha}^{n_k}}{99}\cdot\frac{1}{50{d}^3}<\frac{\lambda_{\alpha}^{n_k}}{40{d}^2},
	\]
	since
	\[
	\frac{100{d}\lambda_{\alpha}^{n_k}}{99}>\frac{100d}{99}\left(\frac{1}{d}-\frac{1}{100d^3}\right)=\frac{100-\tfrac{1}{d^2}}{99}>1.
	\]
	\end{proof}

Now we locate $\tfrac{1}{2}$ in the Rauzy-Veech towers.}
\stepcounter{claimcounter}\begin{claima}\label{cl: 1/2position}
\label{cl:1/2}
For any $k \in \N$, 
	\begin{equation}\label{eq: 1/2position}
	\frac{1}{2}\in \left(T^{\ell_k}(c_\alpha^{n_k})-\frac{\lambda_\alpha^{n_k}}{10}, T^{\ell_k}(c_\alpha^{n_k}) + \frac{\lambda_\alpha^{n_k}}{10}\right).
	\end{equation}
\end{claima}

\begin{proof}[Proof of Claim \ref{cl:1/2}]
Fix $k \in \N$. In view of Lemma \ref{lem: positions_of_centers}, either $\tfrac{1}{2} = T^{\ell_k}(c_\alpha^{n_k})$ or $\tfrac{1}{2}=T^{\ell_k}\left(|I^{n_k}|/2\right)$. In the first case, \eqref{eq: 1/2position} follows trivially. 

In the latter case, since $\tfrac{1}{2}=T^{\ell_k}\left(|I^{n_k}|/2\right) \in T^{\ell_k}(I^{n_k}_{\alpha})$, then $|I^{n_k}|/2 \in I^{n_k}_{\alpha}$ and, by the choice of the set $A$, we have $\pi_0(\alpha) = \tfrac{\#A + 1}{2}$.

Hence, by {Claim} \ref{claim:pos_1/2},
\begin{align*}
\left| c_\alpha^{n_k} - \frac{|I^{n_k}|}{2}\right| & = \frac{1}{2} \big| l_\alpha^{n_k} + r_\alpha^{n_k} - |I^{n_k}| \big| = \frac{1}{2} \Bigg|\sum_{\pi_0(\beta) < \frac{\#A + 1}{2}} \lambda_\beta^{n_k} + \sum_{\pi_0(\beta) \leq \frac{\#A + 1}{2}} \lambda_\beta^{n_k} - \sum_{\beta \in \A} \lambda_\beta^{n_k} \Bigg| \\
& \leq \frac{{d}}{2} \frac{\lambda_\alpha^{n_k}}{{40} {d}^2} < 
\frac{\lambda_\alpha^{n_k}}{10}.
\end{align*}
Since $T^{\ell_k}$ acts via translation on $I^{n_k}_{\alpha}$,
\[ \left| T^{\ell_k}(c_\alpha^{n_k}) - \frac{1}{2} \right| = \left| T^{\ell_k}(c_\alpha^{n_k}) - T^{\ell_k}\left(|I^{n_k}|/2\right) \right| = \left| c_\alpha^{n_k} - \frac{|I^{n_k}|}{2}\right| < \frac{\lambda_\alpha^{n_k}}{10},\]
which implies \eqref{eq: 1/2position}. 
\end{proof}
Let $\overline \alpha = \pi_1^{-1}(\pi_0(\alpha)).$ The symbol $\overline \alpha$ corresponds to the interval positioned symmetrically to the interval indexed by $\alpha$; that is, it satisfies $\pi_0(\overline \alpha) = d + 1 - \pi_0(\alpha)$. For any $k \in \N$, let
\[\h_k:= h^{n_k}_\alpha + h^{n_k}_{\overline \alpha}.\]
We now define a sequence of subintervals of $I^{n_k}$ with good rigidity properties with respect to $T_{n_k}$ and $\h_k$. These sets will be used to define a partial rigidity sequence of sets for $T$. 
\stepcounter{claimcounter}\begin{claima}
\label{cl:rigidity}
For any $k \in \N$, the following holds.
\begin{enumerate}[(i)]
\item \label{cont_interval} $ \left(c^{n_k}_{\alpha} - 
\frac{2\lambda_\alpha^{n_k}}{5}, c^{n_k}_{\alpha} + 
\frac{2\lambda_\alpha^{n_k}}{5}\right)$ is a continuity interval for 
${T_{n_k}^{\pm 2}=}T^{\pm \h_k}.$

\item\label{rigidity_base} ${T_{n_k}^{\pm 1}} \left(c^{n_k}_{\alpha} - 
\frac{2\lambda_\alpha^{n_k}}{5}, c^{n_k}_{\alpha} + 
\frac{2\lambda_\alpha^{n_k}}{5}\right) \subseteq I_{\overline \alpha}^{n_k}$
and ${T_{n_k}^{\pm 2}}\left(c^{n_k}_{\alpha} - 
\frac{2\lambda_\alpha^{n_k}}{5}, c^{n_k}_{\alpha} + 
\frac{2\lambda_\alpha^{n_k}}{5}\right) \subseteq I_\alpha^{n_k}$.

\item\label{far_1/2} $c_\alpha^{n_k} - \tfrac{2\lambda_\alpha^{n_k}}{5} < T^{\h_k}\big(c_\alpha^{n_k} - \tfrac{\lambda_\alpha^{n_k}}{5}\big) < T^{-\ell_k}(\tfrac{1}{2})< T^{\h_k}\big(c_\alpha^{n_k} + \tfrac{\lambda_\alpha^{n_k}}{5}\big) < c_\alpha^{n_k} + \tfrac{2\lambda_\alpha^{n_k}}{5}.$
\end{enumerate}

\end{claima}
\begin{proof}[Proof of Claim \ref{cl:rigidity}]
Let us show that
\begin{equation}
\label{eq:distance_center_alpha}
\left| T_{n_k}(c_{\overline \alpha}^{n_k}) - c_{\alpha}^{n_k}\right| < 
\frac{\lambda_\alpha^{n_k}}{{40}}, \qquad
\left| T_{n_k}(c_{\alpha}^{n_k}) - c_{\overline \alpha}^{n_k}\right| < 
\frac{\lambda_\alpha^{n_k}}{{40}}.
\end{equation}
Since $c_{\overline \alpha}^{n_k}$ is the center of one of the intervals 
exchanged by $T_{n_k}$ and $\pi^{n_k} = \pi$ is symmetric, 
$T_{n_k}(c_{\overline \alpha}^{n_k}) = {|I^{n_k}}| - c_{\overline \alpha}^{n_k} 
$. Thus 
\[T_{n_k}(c_{\overline \alpha}^{n_k}) = \frac{1}{2}({|I^{n_k}|} - l_{\overline 
\alpha}^{n_k}) + \frac{1}{2}({|I^{n_k}|} - r_{\overline \alpha}^{n_k}) = 
\frac{1}{2} \sum_{\pi_0(\beta) \geq \pi_0(\overline \alpha)} 
\lambda_\beta^{n_k} + \frac{1}{2} \sum_{\pi_0(\beta) > \pi_0(\overline \alpha)} 
\lambda_\beta^{n_k} .\]
Since $\pi^{n_k} = \pi$ is symmetric,
\[\#\{\beta \in \A \mid \pi_0(\beta)>\pi_0(\overline \alpha)\}=\#\{\beta \in \A \mid \pi_0(\beta)<\pi_0(\alpha)\}.\]
The previous relations, together with Claim \ref{claim:pos_1/2}, yield
\begin{align*}
\left| T_{n_k}(c_{\overline \alpha}^{n_k}) - c_\alpha^{n_k}\right| & = \frac{1}{2}\left| \sum_{\pi_0(\beta) \geq \pi_0(\overline \alpha)} \lambda_\beta^{n_k} + \sum_{\pi_0(\beta) > \pi_0(\overline \alpha)} \lambda_\beta^{n_k} - \sum_{\pi_0(\beta) < \pi_0(\alpha)} \lambda_\beta^{n_k} - \sum_{\pi_0(\beta) \leq \pi_0(\alpha)} \lambda_\beta^{n_k} \right| \\
& \leq \frac{1}{2}\left| \sum_{\pi_0(\beta) \geq \pi_0(\overline \alpha)} \lambda_\beta^{n_k} - \sum_{\pi_0(\beta) \leq \pi_0(\alpha)} \lambda_\beta^{n_k} \right| + \frac{1}{2}\left| \sum_{\pi_0(\beta) > \pi_0(\overline \alpha)} \lambda_\beta^{n_k} - \sum_{\pi_0(\beta) <\pi_0(\alpha)} \lambda_\beta^{n_k} \right| \\
& \leq \frac{{d}}{2} \frac{\lambda_\alpha^{n_k}}{{20} {d}^2} < 
\frac{\lambda_\alpha^{n_k}}{{40}}.
\end{align*}
This proves the first inequality in \eqref{eq:distance_center_alpha}. The proof of the second inequality is entirely analogous. 

{Let us check that equation \eqref{eq:distance_center_alpha}, together with 
{Claim \ref{claim:pos_1/2} and Claim \ref{cl: 1/2position}}, imply the 
conclusions of the claim. {We first show (\ref{cont_interval}) and 
(\ref{rigidity_base}) for $\h_k$, as the proof for $-\h_k$ follows along the 
same lines. Since 
		\[
	\left(c^{n_k}_{\alpha} - \frac{2\lambda_\alpha^{n_k}}{5}, c^{n_k}_{\alpha} + \frac{2\lambda_\alpha^{n_k}}{5}\right)\subseteq I_{\alpha}^{n_k}	
		\]
	and $T_{n_k}|_{I_{\alpha}^{n_k}}=T^{h_{\alpha}^{n_k}}$, by \eqref{eq:distance_center_alpha}we obtain 
	\[
	\begin{split}
	T^{h_{\alpha}^{n_k}}&\left(c^{n_k}_{\alpha} - \frac{2\lambda_\alpha^{n_k}}{5}, c^{n_k}_{\alpha} + \frac{2\lambda_\alpha^{n_k}}{5}\right)\subseteq \left(c^{n_k}_{\overline\alpha} - \frac{2\lambda_\alpha^{n_k}}{5}-\frac{\lambda_\alpha^{n_k}}{40}, c^{n_k}_{\overline\alpha} + \frac{2\lambda_\alpha^{n_k}}{5}+\frac{\lambda_\alpha^{n_k}}{40}\right)\\
	& =\left(c^{n_k}_{\overline\alpha} - \frac{17\lambda_\alpha^{n_k}}{40}, 
	c^{n_k}_{\overline\alpha} + \frac{17\lambda_\alpha^{n_k}}{40}\right).	
	\end{split}
	\]
	In view of Claim \ref{claim:pos_1/2}, by taking $\beta=\alpha$ and $\beta'=\overline\alpha$, we get 
	\[
	\lambda_\alpha^{n_k}<\frac{40}{39} \lambda_{\overline\alpha}^{n_k}.
	\]
	Thus
	\[
	T^{h_{\alpha}^{n_k}}\left(c^{n_k}_{\alpha} - \frac{2\lambda_\alpha^{n_k}}{5}, c^{n_k}_{\alpha} + \frac{2\lambda_\alpha^{n_k}}{5}\right)\subseteq\left(c^{n_k}_{\overline\alpha} - \frac{17\lambda_{\overline\alpha}^{n_k}}{39}, c^{n_k}_{\overline\alpha} + \frac{17\lambda_{\overline\alpha}^{n_k}}{39}\right)\subseteq I_{\overline\alpha}^{n_k}.
	\]
	
	In particular, we obtained the first inclusion in (\ref{rigidity_base}). Since $I_{\overline\alpha}^{n_k}$ is a continuity interval for $T_{n_k}$ and $T_{n_k}|_{I_{\overline\alpha}^{n_k}}=T^{h_{\overline\alpha}^{n_k}}$, we get (\ref{cont_interval}) . To obtain the second inclusion in (\ref{rigidity_base}) , note that, again by \eqref{eq:distance_center_alpha}, we have
	\[
	T^{h_{\alpha}^{n_k}}\left(c^{n_k}_{\overline\alpha} - 
	\frac{17\lambda_\alpha^{n_k}}{40}, 
	c^{n_k}_{\overline\alpha} + 
	\frac{17\lambda_\alpha^{n_k}}{40}\right)\subseteq 
	\left(c^{n_k}_{\overline\alpha} - \frac{9\lambda_\alpha^{n_k}}{20}, 
	c^{n_k}_{\overline\alpha} + 
	\frac{9\lambda_\alpha^{n_k}}{20}\right)\subseteq I_{\alpha}^{n_k}.
	\]
	
	To prove (\ref{far_1/2}), by Claim \ref{cl: 1/2position}, it is enough to show that 
	\begin{equation}\label{iii_first_ineq}
	c_\alpha^{n_k} - \tfrac{2\lambda_\alpha^{n_k}}{5} < 
	T^{\h_k}\big(c_\alpha^{n_k} - \tfrac{\lambda_\alpha^{n_k}}{5}\big) 
	<c_\alpha^{n_k} - \tfrac{\lambda_\alpha^{n_k}}{10} 
	\end{equation}
	and
		\[
	c_\alpha^{n_k} + \tfrac{\lambda_\alpha^{n_k}}{10} < 
	T^{\h_k}\big(c_\alpha^{n_k} - \tfrac{\lambda_\alpha^{n_k}}{5}\big) 
	<c_\alpha^{n_k} + \tfrac{2\lambda_\alpha^{n_k}}{5}. 
	\]
	We will show the first inequality; the latter is proven analogously. By 
	\eqref{eq:distance_center_alpha} we have
	\[
	T^{\h_k}\big(c_\alpha^{n_k} - \tfrac{\lambda_\alpha^{n_k}}{5}\big)= 
	T_{n_k}^2\big(c_\alpha^{n_k} - \tfrac{\lambda_\alpha^{n_k}}{5}\big)>
	\big(c_\alpha^{n_k} - \tfrac{\lambda_\alpha^{n_k}}{5} - 
	\tfrac{\lambda_\alpha^{n_k}}{20}\big)>\big(c_\alpha^{n_k} - 
	\tfrac{2\lambda_\alpha^{n_k}}{5}\big)
	\]
	and
	\[
	T^{\h_k}\big(c_\alpha^{n_k} - \tfrac{\lambda_\alpha^{n_k}}{5}\big)= 
	T_{n_k}^2\big(c_\alpha^{n_k} - \tfrac{\lambda_\alpha^{n_k}}{5}\big)<
	\big(c_\alpha^{n_k} - \tfrac{\lambda_\alpha^{n_k}}{5} + 
	\tfrac{\lambda_\alpha^{n_k}}{20}\big)<\big(c_\alpha^{n_k} - 
	\tfrac{\lambda_\alpha^{n_k}}{10}\big),
	\]
	which proves \eqref{iii_first_ineq}.

}	
}
\end{proof}

By taking preimages of an appropriate subinterval of the floor $T^{\ell_k}(I^{n_k}_\alpha)$, we can construct a set on which we have good control of the $\h_k$-th Birkhoff sum of $f$. Denote
\begin{equation}
\label{eq:min_heights}
q_k:= \min_{\alpha \in \A}h^{n_k}_\alpha,
\end{equation}
for any $k \geq 1$. 
\stepcounter{claimcounter}\begin{claima}
\label{cl:controled_BS}
Let $k \in \N$. For any $x \in T^{-i} \left( T^{\ell_k} \left(c^{n_k}_{\alpha} - \frac{2\lambda_\alpha^{n_k}}{5}, c^{n_k}_{\alpha} + \frac{2\lambda_\alpha^{n_k}}{5}\right)\right)$ and any $0 \leq i < q_k$,
\begin{equation}
\label{eq:constant_BS}
S_{\h_k} f(x) - S_{\h_k} f(\tfrac{1}{2} ) = f(T^{i}(x)) = \left\{ \begin{array}{lcr} 1 & \text{ if } & T^{i}(x) < \tfrac{1}{2}, \\ 0 & \text{ if } & T^{i}(x) = \tfrac{1}{2}, \\ -1 & \text{ if } & T^{i}(x) > \tfrac{1}{2}. \\ \end{array}\right.
\end{equation}
\end{claima}

\begin{proof}
By \eqref{cont_interval} and \eqref{rigidity_base} in Claim \ref{cl:rigidity}, given $x \in T^{-i} \left(T^{\ell_k}(c_\alpha^{n_k}) - \frac{2\lambda_\alpha^{n_k}}{5}, T^{\ell_k}(c_\alpha^{n_k}) + \frac{2\lambda_\alpha^{n_k}}{5}\right)$ for some $0 \leq i < q_k$, there is exactly one point of $\{x, T(x), \dots, T^{\h_k - 1}(x)\}$ in each of the levels of the towers $\bigcup_{j=0}^{h_{\alpha}^{n_k}-1}T^j(I^{n_k}_\alpha)$ and $\bigcup_{j=0}^{h_{\overline \alpha}^{n_k}-1}T^j(I^{n_k}_{\overline \alpha}).$ 

Since $\tfrac{1}{2}$ is the only discontinuity of $f$ and $T^{i}(x), 
\tfrac{1}{2} \in T^{\ell_k}(I^{n_k}_\alpha)$, 
\[S_{\h_k} f(x) - S_{\h_k} f(\tfrac{1}{2} ) = f(T^{i}(x)) - f(\tfrac{1}{2}) = f(T^{i}(x)),\]
{and the proof of the claim follows from the form of $f$}.
\end{proof}

Taking the previous claims into account, we define a sequence of Rohlin towers $(\Xi_k)_{k \in \N}$ by 
\[ \Xi_k := \bigcup_{i=0}^{q_k-1}T^{-i}(B_k),\qquad B_k:= 
T^{\ell_k} \left(c^{n_k}_{\alpha} + \frac{\lambda_\alpha^{n_k}}{5}, c^{n_k}_{\alpha} + \frac{2\lambda_\alpha^{n_k}}{5}\right), \]
for any $k\in\N$, where $q_k$ is given by \eqref{eq:min_heights}. Notice that we choose as a base for the Rohlin tower a subinterval to the right of $\tfrac{1}{2}$ and that $\Xi_k$ is contained in the union of at most two Rohlin towers. This choice is made so that, for any $x \in \Xi_k$, the value in the RHS of \eqref{eq:constant_BS} equals $-1$.

\begin{figure}[h]
\begin{subfigure}{.5\textwidth}
 \centering
 \includegraphics[scale=0.73]{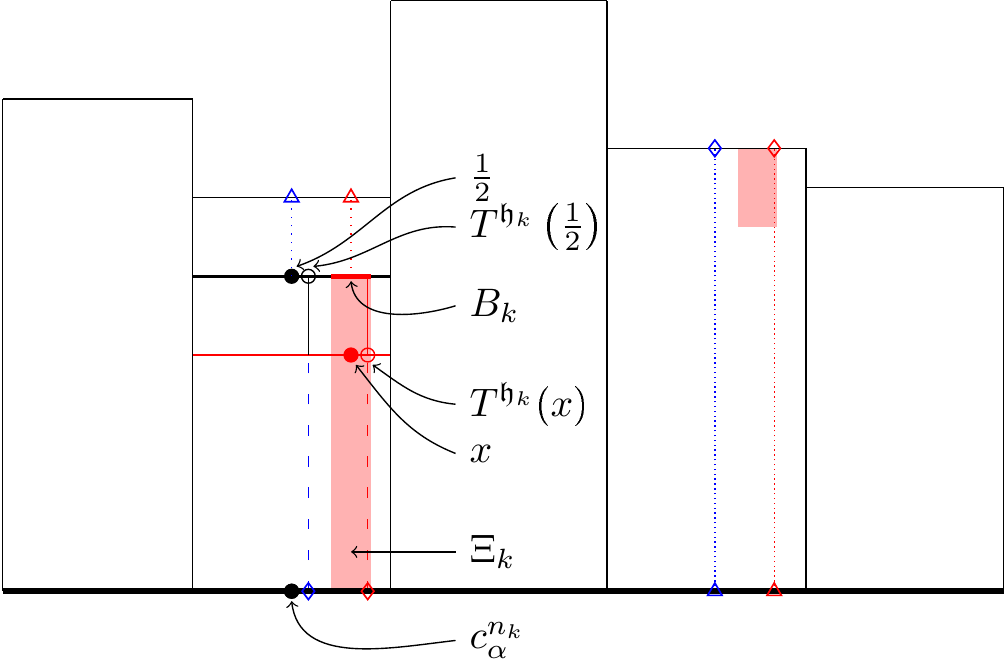}
 \caption{$\alpha =\pi_0^{-1}(2)$.}
 \label{fig:sub1}
\end{subfigure}%
\bigskip
\begin{subfigure}{.5\textwidth}
 \centering
 \includegraphics[scale=0.73]{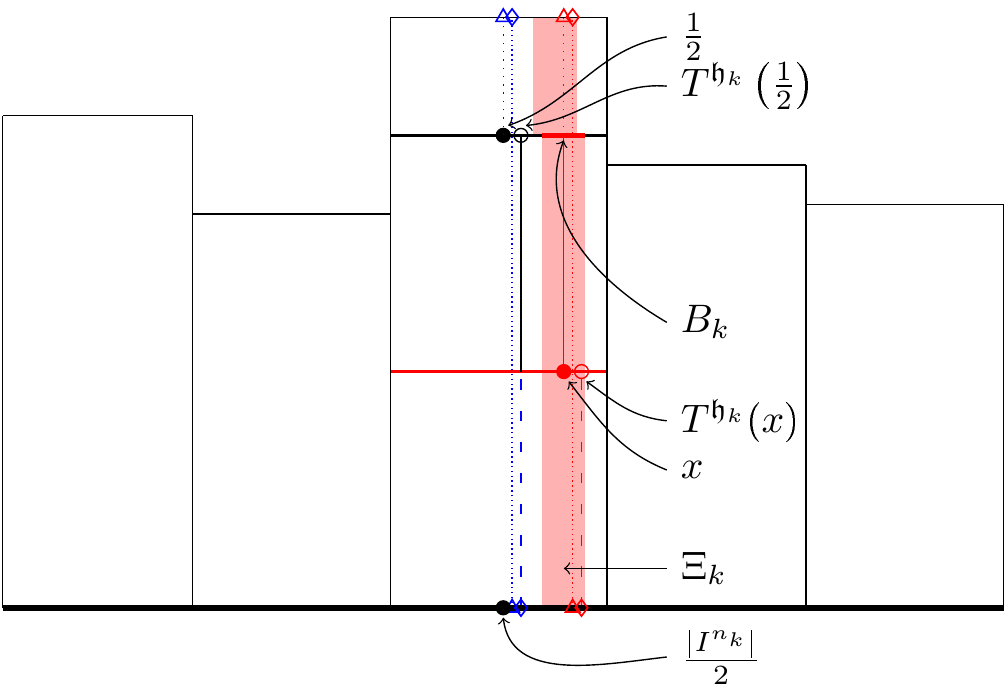}
 \caption{$\alpha =\pi_0^{-1}(\tfrac{d + 1}{2})$ and $\tfrac{1}{2}=T^{\ell_k}\left(|I^{n_k}|/2\right)$.}
 \label{fig:sub2}
\end{subfigure}
\caption{\small Case A: $d = \#\A$ odd. Through different dashing patterns, the figure above shows how to split the Birkhoff sums $S_{\mathfrak{h}_k}f(\tfrac{1}{2})$ and $S_{\mathfrak{h}_k}f(x)$ into pieces having the same value, except for the points belonging to the floor that contains $\tfrac{1}{2}$. As $\tfrac{1}{2}$ is the only discontinuity of $f$, if $\tfrac{1}{2}$ is not in the forward orbit of $x$ these $\mathfrak{h}_k$-th Birkhoff sums differ by $\pm1$.}
\label{fig:odd}
\end{figure}

Let us show that $(\Xi_k)_{k \in \N}$ and $(\h_k)_{k \in \N}$ define a partial rigidity sequence for $T$ satisfying
 \begin{equation}
 \label{eq:BS_double}
 S_{\h_k}f\mid_{\Xi_k} = {-1}.
 \end{equation}

Note that $\textup{Leb}(\Xi_k)=q_k\cdot\lambda_{\alpha}^{n_k}$. Since the Rauzy 
path 
$\gamma$ given by Lemma \ref{lem: positions_of_centers} is 
{positive} and fixed, Proposition \ref{prop: gammaconstant} together with the 
choice of the set $A$ yield
\[ 1 = \sum_{\beta \in \A} h_\beta^{n_k} \lambda_\beta^{n_k} \leq 2{d} 
\rho(\gamma) q_k \lambda _\alpha^{n_k} ={10d} 
\rho(\gamma)\textup{Leb}(\Xi_k),\] 
for any $k \in \N$, where $\rho(\gamma) > 0$ is the constant given by Proposition \ref{prop: gammaconstant}. Hence, up to taking a subsequence, we may assume that
 \[ \delta:= \lim_{k \to \infty} \textup{Leb}(\Xi_k) > 0.\]
 As $|I^{n_k}|\to 0$ as $k\to\infty$, the previous equation together with \eqref{rigidity_base} in Claim \ref{cl:rigidity} show that $(\Xi_k)_{k\in\N}$ is partial rigidity sequence for $T$ with rigidity times $(\h_k)_{k\in\N}$. 
 
Fix $k \in \N$, and let us show that \eqref{eq:BS_double} holds. Notice that by 
\eqref{far_1/2} in Claim \ref{cl:rigidity} and {Claim} \ref{cl:controled_BS}, 
$S_{\h_k}f\mid_{\Xi_k}$ is constant and satisfies
\begin{equation}
\label{eq:BS_formula}
S_{\h_k}f\mid_{\Xi_k} = S_{\h_k} f (\tfrac{1}{2}) - 1.
\end{equation}
By \eqref{eq: 1/2position} and \eqref{far_1/2} in Claim \ref{cl:rigidity}, $$T^{-\h_k}(\tfrac{1}{2}), \tfrac{1}{2}, T^{\h_k}(\tfrac{1}{2}) \in T^{\ell_k} \left(c^{n_k}_{\alpha} - \frac{2\lambda_\alpha^{n_k}}{5}, c^{n_k}_{\alpha} + \frac{2\lambda_\alpha^{n_k}}{5}\right).$$ By Corollary \ref{cor: cancellations} and Claim \ref{cl:controled_BS} applied to $T^{-\h_k}(\tfrac{1}{2})$, 
\[ - f(T^{\h_k}(\tfrac{1}{2})) = S_{2\h_k} f (T^{-\h_k}(\tfrac{1}{2})) = S_{\h_k} f (T^{-\h_k}(\tfrac{1}{2})) + S_{\h_k} f (\tfrac{1}{2}) = 2S_{\h_k} f (\tfrac{1}{2}) + f (T^{-\h_k}(\tfrac{1}{2})).\]
Hence, since either $T^{-\h_k}(\tfrac{1}{2}) < \tfrac{1}{2} < T^{\h_k}(\tfrac{1}{2})$ or $T^{\h_k}(\tfrac{1}{2}) < \tfrac{1}{2} < T^{-\h_k}(\tfrac{1}{2}),$
\begin{equation}\label{eq: valueat1/2}
	 S_{\h_k} f (\tfrac{1}{2}) = \frac{1}{2}\big( - f (T^{\h_k}(\tfrac{1}{2})) - f (T^{-\h_k}(\tfrac{1}{2}))\big) = 0.
\end{equation}

Therefore, by \eqref{eq:BS_formula}, equation \eqref{eq:BS_double} follows. 

It follows from Proposition \ref{prop:intessentialvalue} that $-1$ is an 
essential value for $T_f$, and thus, by Theorem \ref{thm: erg_criterion}, $T_f$ 
is ergodic. This finishes the proof for ${d}$ odd.

\medskip

\noindent \textbf{Case B: ${d}$ is even.} {We illustrate this situation in 
Figure \ref{fig:even}. Notice that, in this case, not all intervals are of 
comparable size.} {However, as the proof follows the same arguments as in Case 
A, we will not provide a full proof but rather give the analogous statements 
for the Claims \ref{claim:pos_1/2}-\ref{cl:controled_BS} that will allow to 
conclude the proof.} 

\begin{figure}[h]
\begin{subfigure}{.5\textwidth}
 \centering
 \includegraphics[scale=1]{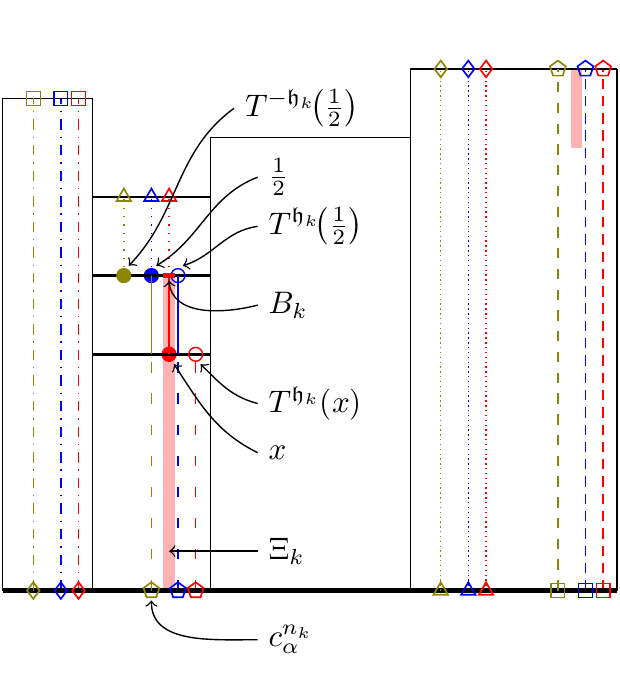}
 \caption{$\alpha =\pi_0^{-1}(2)$.}
 \label{fig:sub3}
\end{subfigure}%
\begin{subfigure}{.5\textwidth}
 \centering
 \includegraphics[scale=1]{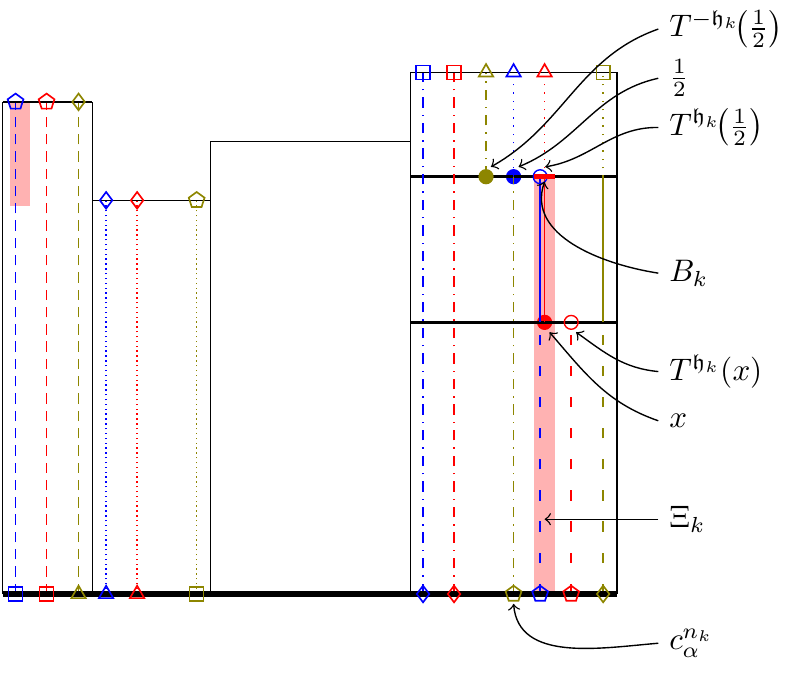}
 \caption{$\alpha =\pi_0^{-1}(d)$.}
 \label{fig:sub4}
\end{subfigure}

\caption{\small Case B: $d = \#\A$ even. Through different dashing patterns, the figure above shows how to split the Birkhoff sums $S_{\mathfrak{h}_k}f(T^{-\mathfrak{h}_k}(\tfrac{1}{2}))$, $S_{\mathfrak{h}_k}f(\tfrac{1}{2})$ and $S_{\mathfrak{h}_k}f(x)$ into pieces having the same value, except for the points belonging to the floor that contains $\tfrac{1}{2}$. As $\tfrac{1}{2}$ is the only discontinuity of $f$, if $\tfrac{1}{2}$ is not in the forward orbit of $x$, these $\mathfrak{h}_k$-th Birkhoff sums differ by $\pm1$.}
\label{fig:even}
\end{figure}

{An analogue of Claim \ref{claim:pos_1/2} is as follows and its proof follows 
precisely the same lines as the proof of Claim \ref{claim:pos_1/2}. 
	\setcounter{claimcounter}{0}
	\stepcounter{claimcounter}\begin{claimb}\label{cl:newclaim1}
	For every $k\in\N$, we have
	\begin{equation*}
		\begin{split}
	\max\Big\{&	\max_{\beta,\beta'\in\mathcal 
	A\setminus\{\pi_0^{-1}(1),\pi_0^{-1}(2) 
	\}}\left|\lambda_{\beta}^{n_k}-\lambda_{\beta'}^{n_k}\right|,\, 
	\max_{\beta\in\mathcal A\setminus\{\pi_0^{-1}(1),\pi_0^{-1}(2) 
	\}}\left|\lambda_{\beta}^{n_k}-\tfrac{7}{3}\lambda_{\pi_0^{-1}(1)}^{n_k}\right|,\\ 
	&\max_{\beta\in\mathcal A\setminus\{\pi_0^{-1}(1),\pi_0^{-1}(2) 
		\}}\left|\lambda_{\beta}^{n_k}
	-\tfrac{7}{4}\lambda_{\pi_0^{-1}(2)}^{n_k}\right|,\left|\tfrac{7}{3}\lambda_{\pi_0^{-1}(1)}^{n_k}-\tfrac{7}{4}\lambda_{\pi_0^{-1}(2)}^{n_k}\right|
	\Big\}<\frac{\lambda_{\alpha}^{n_k}}{{40}{d}^2}.
	\end{split}
	\end{equation*}
	\end{claimb}

If $\alpha\notin \{\pi_0^{-1}(1),\pi_0^{-1}(2),\pi_0^{-1}(d) \}$, the rest of 
the proof of Theorem \ref{thm: ergodicity1} in Case B follows as in Case A by 
using Claim \ref{cl:newclaim1} instead of Claim \ref{claim:pos_1/2}. The only 
difference lies in the choice of $\overline\alpha \in \A$, that is, we pick 
$\pi_0(\overline\alpha)=d+2-\pi_0(\alpha)$. {From now on we focus on the case 
when $\alpha \in \{\pi_0^{-1}(1),\pi_0^{-1}(2),\pi_0^{-1}(d) \}$.}
} 

	
	{For $\alpha \in \{\pi_0^{-1}(1),\pi_0^{-1}(2),\pi_0^{-1}(d) \}$, Claim 
	\ref{cl: 1/2position} remains unchanged. 
	}
	 
Take 
\[
\mathfrak{h}_k:=h_{\pi_0^{-1}(1)}^{n_k}+h_{\pi_0^{-1}(2)}^{n_k}
+2h_{\pi_0^{-1}(d)}^{n_k}.
\]
It turns out, as evidenced by the point (ii) in Claim \ref{cl:B_iterates} below, that this choice is universal for all $\alpha\in\{\pi_0^{-1}(1),\pi_0^{-1}(2),\pi_0^{-1}(d) \}$.

{Claim \ref{cl:rigidity} is not true in Case B if $\alpha\in 
\{\pi_0^{-1}(1),\pi_0^{-1}(2),\pi_0^{-1}(d) \}$. The main reason is that the 
intervals of continuity of $T_{n_k}^{\pm i}$, where $i=-4,\ldots,4$, around the 
intervals' centers are much smaller than in Case A. The following is an 
alternative to Claim \ref{cl:rigidity}.}

	\stepcounter{claimcounter}\stepcounter{claimcounter}\begin{claimb}\label{cl:B_iterates}
	For any $\alpha\in \{\pi_0^{-1}(1),\pi_0^{-1}(2),\pi_0^{-1}(d) \}$ and $k \in \N$, the following holds.
	\begin{enumerate}[(i)]
		\item \label{cont_interval2} $ \left(c^{n_k}_{\alpha} - 
		\frac{\lambda_\alpha^{n_k}}{50}, c^{n_k}_{\alpha} + 
		\frac{\lambda_\alpha^{n_k}}{50}\right)$ is a continuity interval for 
		${T_{n_k}^{\pm 4}=}T^{\pm \h_k}.$
		\item\label{rigidity_base2} $T_{n_k}^{i} \left(c^{n_k}_{\alpha} - \frac{\lambda_\alpha^{n_k}}{50}, c^{n_k}_{\alpha} + \frac{\lambda_\alpha^{n_k}}{50}\right) \subseteq I_{\pi_0^{-1}(b(i,\alpha))}^{n_k}$, where $i=-4,\dots, 4$, and 
		\begin{itemize}
			\item $[b(i,\alpha)]_{i=-4}^4=[1,d,2,d,1,d,2,d,1]$, if $\alpha=\pi_0^{-1}(1)$,
			\item $[b(i,\alpha)]_{i=-4}^4=[2,d,1,d,2,d,1,d,2]$, if $\alpha=\pi_0^{-1}(2)$,
			\item $[b(i,\alpha)]_{i=-4}^4=[d,1,d,2,d,2,d,1,d]$, if $\alpha=\pi_0^{-1}(d)$.
		\end{itemize}
		\item\label{far_1/22} $c_\alpha^{n_k} - \tfrac{\lambda_\alpha^{n_k}}{50} < 
		T^{-\ell_k}(\tfrac{1}{2})=c_{\alpha}^{n_k}< 
		c_\alpha^{n_k} + \tfrac{\lambda_\alpha^{n_k}}{50}.$
	\end{enumerate}
	\end{claimb}
		
{Claim \ref{cl:controled_BS} changes slightly due to the difference in the 
lengths of the first two intervals in the definition of the set $A$. 

	\stepcounter{claimcounter}\begin{claimb}
\label{cl:controled_BS_2}
Let $k \in \N$ and $\alpha\in \{\pi_0^{-1}(1),\pi_0^{-1}(2),\pi_0^{-1}(d) \}$. For any $x \in T^{-i} \left( T^{\ell_k} \left(c^{n_k}_{\alpha} - \frac{\lambda_\alpha^{n_k}}{50}, c^{n_k}_{\alpha} + \frac{\lambda_\alpha^{n_k}}{50}\right)\right)$ and any $0 \leq i < q_k$,
\begin{equation*}
S_{\h_k} f(x) - S_{\h_k} f(\tfrac{1}{2} ) = f(T^{i}(x)) = \left\{ \begin{array}{lcr} 1 & \text{ if } & T^{i}(x) < \tfrac{1}{2}, \\ 0 & \text{ if } & T^{i}(x) = \tfrac{1}{2}, \\ -1 & \text{ if } & T^{i}(x) > \tfrac{1}{2}. \\ \end{array}\right.
\end{equation*}
\end{claimb}

The proof of this claim remains essentially 
the same as that of Claim \ref{cl:controled_BS}. Indeed, it is enough to break 
the Birkhoff sums into pieces corresponding to each of the towers and compare 
the values of Birkhoff sums within said towers (see Figure \ref{fig:even} and 
the matching of orbits for the points $\tfrac{1}{2}$ and $T^{\pm\mathfrak 
h_k}(\tfrac{1}{2})$). {More precisely, for any $x\in\left(c^{n_k}_{\alpha} - 
\frac{\lambda_\alpha^{n_k}}{50}, c^{n_k}_{\alpha} + 
\frac{\lambda_\alpha^{n_k}}{50}\right)$ and any $i=0,\ldots,4$, it follows 
easily from the definition of $A$ that $T_{n_k}^i(x)\in 
I^{n_k}_{\pi_0^{-1}(b(i,\alpha)) }$, where
\begin{itemize}
	\item $[b(i,\alpha)]_{i=0}^4=[1,d,2,d,1]$ if $\alpha=\pi_0^{-1}(1)$,
	\item $[b(i,\alpha)]_{i=0}^4=[2,d,1,d,2]$ if $\alpha=\pi_0^{-1}(2)$,
	\item $[b(i,\alpha)]_{i=0}^4=[d,2,d,1,d]$ 
	if $\alpha=\pi_0^{-1}(d)$.
\end{itemize}}
It is important to notice that while the above claim is enough to construct the tower $\Xi_k$ and compare the Birkhoff sums of the points inside with $S_{\mathfrak h_k}f(\tfrac{1}{2})$, it is not enough to deduce the equation \eqref{eq: valueat1/2} in this case, since the point $T^{-\mathfrak h_k}(\tfrac{1}{2})$ does not belong to the interval $T^{\ell_k} \left(c^{n_k}_{\alpha} - \frac{\lambda_\alpha^{n_k}}{50}, c^{n_k}_{\alpha} + \frac{\lambda_\alpha^{n_k}}{50}\right)$. However, since $\tfrac{1}{2}\in T^{\ell_k} \left(c^{n_k}_{\alpha} - \frac{\lambda_\alpha^{n_k}}{50}, c^{n_k}_{\alpha} + \frac{\lambda_\alpha^{n_k}}{50}\right)$, by (ii) in Claim \ref{cl:B_iterates} we get for every $i=0,\ldots,4$ that $T_{n_k}^i\left(T^{-\ell_k}\left(T^{-\mathfrak h_k}(\tfrac{1}{2})\right)\right)\in I^{n_k}_{\pi_0^{-1}(b(i-4,\alpha))}$, where
\begin{itemize}
	\item $[b(i-4,\alpha)]_{i=0}^4=[1,d,2,d,1]$ if $\alpha=\pi_0^{-1}(1)$,
	\item $[b(i-4,\alpha)]_{i=0}^4=[2,d,1,d,2]$ if $\alpha=\pi_0^{-1}(2)$,
	\item $[b(i-4,\alpha)]_{i=0}^4=[d,1,d,2,d]$ 
	if $\alpha=\pi_0^{-1}(d)$.
\end{itemize}}
The first two of the above itineraries via $T_{n_k}$ are identical as that of all points $x\in\left(c^{n_k}_{\alpha} - \frac{\lambda_\alpha^{n_k}}{50}, c^{n_k}_{\alpha} + \frac{\lambda_\alpha^{n_k}}{50}\right)$, in particular $T^{-\ell_k}(\tfrac{1}{2})$. Thus, the Birkhoff sums in these cases can be controlled 
directly as in the proof of \eqref{eq: valueat1/2} in Case A. The last 
itinerary is different, but what is relevant is that the orbits of length 
$\mathfrak{h}_k$ of points $\frac{1}{2}$ and $T^{-\mathfrak h_k}(\tfrac{1}{2})$ 
travel through the same towers, just in a different order. That is enough to 
compare the associated Birkhoff sums, again because $f$ is 
piecewise constant and $\frac{1}{2}$ is its only discontinuity.

{Finally, using the facts listed above and following the same arguments as in 
Case A, one can show that
	\[ \Xi_k := \bigcup_{i=0}^{q_k-1}T^{-i}(B_k),\qquad B_k:= 
	T^{\ell_k} \left(c^{n_k}_{\alpha} + \frac{\lambda_\alpha^{n_k}}{100}, c^{n_k}_{\alpha} + \frac{\lambda_\alpha^{n_k}}{50}\right), \]
	where $q_k$ is given by \eqref{eq:min_heights}, is a partial rigidity 
	sequence for $T$ satisfying $S_{\h_k}f\mid_{\Xi_k} = {-1}.$ Proposition 
	\ref{prop:intessentialvalue} and Theorem \ref{thm: erg_criterion} then 
	guarantee that $T_f$ is ergodic.}
\vspace{3mm}


\emph{Acknowledgements:} Both authors would like to thank Corinna Ulcigrai for 
her constant support and numerous discussions on the subject. We would also 
like to thank Jamerson Bezerra, Davide Ravotti{, and anonymous reviewers} for 
helping to improve the quality of the paper. The first author would also like 
to thank the organizers and the speakers of the summer school ``Geometry and 
Dynamics of Moduli Spaces'', which took place in Bologna in 2022, since our 
result partially solves the problem posed there during the open problems 
session. Finally, both authors acknowledge the support of the {\it Swiss 
National Science Foundation} through Grant $200021\_188617/1$. The first author 
also thanks the {\it National Science Centre (Poland)} Grant OPUS 
2022/45/B/ST1/00179. The second author was partially supported by the UZH 
Postdoc Grant, grant no. FK-23-133, during the preparation of this work.
\bibliographystyle{acm}
\bibliography{Bibliography.bib}
\end{document}